\documentclass[a4paper,10pt]{amsart}
\usepackage{amsmath,amssymb,graphicx,url}

\newtheorem{theorem}{Theorem}[section]
\newtheorem{prop}[theorem]{Proposition}
\newtheorem{coro}[theorem]{Corollary}
\newtheorem{lemma}[theorem]{Lemma}

\theoremstyle{definition}

\newtheorem{rem}[theorem]{Remark}

\begin{document}
\title[]{An Extended symmetric union and its Alexander polynomial}
\author{Teruaki KITANO \and Yasuharu NAKAE}
\address{Department of Information Systems Science, Faculty of Science and Engineering, Soka University, Tangi-machi 1-236, Hachioji, Tokyo, 192-8577, Japan}
\email{kitano@soka.ac.jp}
\address{Graduate School of Engineering Science, Akita University,
1-1 Tegata Gakuen-machi, Akita city, Akita, 010-8502, Japan.}
\email{nakae@math.akita-u.ac.jp}
\thanks{
\textit{Keyword and phrases:} Knot group, epimorphism, symmetric union, Alexander polynomial
}
\subjclass[2020]{Primary 57K10; Secondary 57M05}

\begin{abstract}
For prime knots $K_1$ and $K_2$, we write $K_1 \geq K_2$ if there is an epimorphism from the knot group of $K_1$ to that of $K_2$ which preserves the meridian.
We construct a family of pairs of knots with $K_1 \geq K_2$ such that an epimorphism maps the longitude of $K_1$ to the trivial element.
This construction is regarded as an extension of a symmetric union with a single full twisted region.
In particular, it extends a property of the Alexander polynomial of a symmetric union.
We also exhibit that all but two of the knots up to ten crossings in the list of Kitano-Suzuki, which have an epimorphism mapping the longitude to the trivial element, arise from this construction.
\end{abstract}

\maketitle

\setlength{\baselineskip}{14pt}

\section{Introduction}
Let $K$ be a knot in $S^3$, and let $G(K)=\pi_1(S^3\setminus K)$ denote the knot group of $K$.
For prime knots $K_1$ and $K_2$, we write $K_1 \geq K_2$ if there is an epimorphism $\varphi: G(K_1)\to G(K_2)$
and $\varphi$ preserves the meridian.
The relation $\geq$ becomes a partial order on the set of prime knots, see \cite{KS2008}.
The list of pairs of knots with such epimorphisms is studied by Kitano and Suzuki \cite{KS2005, KS2005C},
and by Horie, Kitano, Matsumoto, and Suzuki \cite{HKMS2011,HKMS2011E}.
We study the construction of a pair of knots with an epimorphism between their knot groups,
especially in cases when this epimorphism maps the longitude to the trivial element.
We construct such a family of pairs of knots 
and derive the properties of the Alexander polynomials as follows.

Let $D$ be a planar diagram of a knot $\hat{K}$ in $S^3$.
We take the diagram $D^\ast$ as the diagram of the knot obtained by reflecting $\hat{K}$ across the plane orthogonal to the plane on which the diagram $D$ lies.  
Let $T$ be a tangle such that the left-side endpoints are connected in $T$, and the right-side endpoints are also connected in $T$. 
We define a knot $K$ constructed from the diagram in Figure \ref{maintheorem_diagram}.

\begin{figure}[h]
	\centerline{\includegraphics[keepaspectratio]{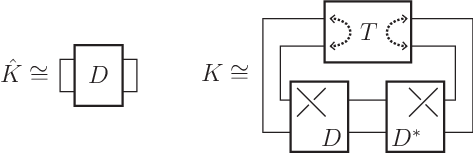}}
\caption{the diagram of a knot $K$}
\label{maintheorem_diagram}
\end{figure}

\begin{theorem}\label{maintheorem}

The knot $K$, constructed as described above, satisfies the following properties.
\begin{enumerate}
\item $\Delta_K(t)=\Delta_{K'}(t)\left(\Delta_{\hat{K}}(t)\right)^2$, where $K'$ be the numerator $N(T)$ of $T$.
\item There is an epimorphism $\varphi:G(K)\to G(\hat{K})$ such that
$\varphi$ maps a meridian of $K$ to a meridian of $\hat{K}$,
and maps the preferred longitude of $K$ to the trivial element in $G(\hat{K})$.
\end{enumerate}
\end{theorem}

Where for a diagram $D$ with four endpoints arranged at the corners,
the numerator $N(D)$ of $D$ is constructed by connecting the upper endpoints to each other
and by also connecting the lower endpoints.
A simple closed curve on the boundary of the tubular neighborhood of a knot is referred to as the preferred longitude if it bounds a Seifert surface.
In this paper, a longitude refers to a preferred longitude.

We shall prove Theorem \ref{maintheorem} in Section 2 and exhibit examples constructed from the framework of Theorem \ref{maintheorem} in Section 3.

Kinoshita and Terasaka introduced a symmetric union of knots \cite{KT1957},
and Lamm generalized the definition and their results \cite{L2000}.
For the general definition of a symmetric union of knots, see \cite{L2000}.

In \cite{KT1957},
Kinoshita and Terasaka showed that a symmetric union of a partial knot $\hat{K}$ has the following property regarding its Alexander polynomial.

\begin{theorem}\label{symunionalexander}\cite[Theorem 2]{KT1957}
Let $K$ be a symmetric union of a partial knot $\hat{K}$ with a single full-twisted region.
Then $\Delta_K(t)=\left(\Delta_{\hat{K}}(t)\right)^2$.
\end{theorem}

Eisermann showed that a symmetric union of knots with multiple full-twisted regions has the following property regarding the existence of an epimorphism that maps the longitude to the trivial element.

\begin{theorem}\label{symunionlongitude}\cite[Theorem 3.3]{L2000}
Let $K$ be a symmetric union of a partial knot $\hat{K}$ such that the twisting regions of $K$ are all full-twists.
Then there is an epimorphism $\varphi: G(K)\to G(\hat{K})$ that maps the meridian of $K$ to the meridian of $\hat{K}$ and maps the longitude of $K$ to the trivial element.
\end{theorem}

In Theorem \ref{maintheorem},
if the tangle $T$ is equivalent to $n$ twists, namely $T=T(1/n)$ for some non-zero even integer $n$,
the knot constructed by way of Theorem \ref{maintheorem} is a symmetric union of the partial knot $\hat{K}$
with a single full-twisted region.
Therefore, the consequences of Theorem \ref{maintheorem} inherit and extend the properties of a symmetric union of knots.

Theorem \ref{maintheorem} holds even if the partial knot $\hat{K}$ is trivial and the diagram $D$ induced by $\hat{K}$ is also trivial.
By taking a non-trivial diagram $D$ induced by a trivial knot $\hat{K}$,
one can construct the so-called Kinoshita-Terasaka knot \cite{KT1957}, whose Alexander polynomial is trivial,
as the symmetric union of $\hat{K}$ with the twisting region equivalent to $T=(1/2)$.

\begin{rem}
A similar construction of knots in Theorem \ref{maintheorem} appears in \cite{IMS2015} as an RTR knot.
But in our construction, the tangles $D$ and $T$ are not necessarily rational tangles.
\end{rem}

The existence of an epimorphism between knot groups leads to some interesting properties among these knots.

If $K_1 \geq K_2$, the Alexander polynomial $\Delta_{K_1}(t)$ is divisible by $\Delta_{K_2}(t)$,
so it is clearly satisfied for the knots $K$ and a partial knot $\hat{K}$ in Theorem \ref{maintheorem}.

If $K_1 \geq K_2$ and $K_1$ is a fibered knot, then $K_2$ is also fibered \cite[Proposition 2.2]{KS2008}.
But the converse is not true.
For example, $10_{65}\geq 3_1$ in the list of \cite{KS2005}, and the trefoil $3_1$ is fibered, but $10_{65}$ is not fibered.
This can be easily concluded from Theorem \ref{maintheorem}, as follows.
For the knot  $10_{65}$,
we can see $\hat{K}=3_1$ and $N(T)=5_2$
as the knot constructed using Theorem \ref{maintheorem}
(this construction will be explained in Section 3).
The Alexander polynomial of $5_2$ is $\Delta_{5_2}(t)=2t^2-3t+2$, and that of the trefoil is $\Delta_{3_1}(t)=t^2-t+1$.
Then $\Delta_{10_{65}}(t)=\Delta_{5_2}(t)\cdot(\Delta_{3_1}(t))^2$ is not monic.
Therefore, we can see $10_{65}$ is not fibered since the Alexander polynomial of a fibered knot must be monic.

If there is a degree one map $f:(E(K_1), \partial E(K_1))\to (E(K_2), \partial E(K_2))$, the induced map $f_\ast : G(K_1)\to G(K_2)$ is an epimorphism \cite[Proposition 3]{BBRW2016}.
If $K_1 \geq K_2$, $K_1$ is a hyperbolic 2-bridge knot, and $K_2$ is a non-trivial knot, 
then the epimorphism $\varphi : G(K_1) \to G(K_2)$ is induced from a non-zero degree map,
and $K_2$ is also a 2-bridge knot \cite[Corollary 1.3]{BBRW2010}.
In Theorem \ref{maintheorem},
since the epimorphism $\varphi$ preserves the meridian and maps the longitude to the trivial element,
a map $f:(E(K), \partial E(K))\to (E(\hat{K}), \partial E(\hat{K}))$ which induces the epimorphism $\varphi$
is of degree zero  (cf.\,\cite{KS2008}).
Focusing on 2-bridge knots and Montesinos knots,
we can derive the following corollary from Theorem \ref{maintheorem}.
For the definitions of a rational tangle and a Montesinos knot, see \cite{BZ3rd}, \cite{HM2006}.

\begin{coro}\label{coro1.5}
Let $\hat{K}$ be a fibered 2-bridge knot.
Then there are infinitely many non-fibered Montesinos knots $K$
such that $K\geq \hat{K}$ and the longitude of $K$ is mapped to the trivial element.
\end{coro}

The proof of Corollary \ref{coro1.5} is explained in Section 2.3.

\begin{rem}
All of the non-fibered Montesinos knots $K$ constructed in Corollary \ref{coro1.5} are not 2-bridge knots.
By \cite[Corollary 1.3]{BBRW2010} explained before,
$K$ cannot be a 2-bridge knot because the partial knot $\hat{K}$ is a 2-bridge knot and the epimorphism $\varphi$ is induced from a zero-degree map.
\end{rem}

\section{Proof of Main Theorem}
This section is organized into three subsections for the proofs of (1) and (2) of Theorem \ref{maintheorem}, as well as the proof of Corollary \ref{coro1.5}.

\subsection{Proof of Theorem \ref{maintheorem} (1)}

$\nabla_K(z)$ denotes the Conway polynomial of a knot $K$.
In order to prove Theorem \ref{maintheorem} (1),
we will show that
$
\nabla_K(z)
=\nabla_{N(\tilde{T})}(z)\left(\nabla_{\hat{K}}(z)\right)^2
$,
where $\tilde{T}$ is the tangle obatained by rotating $T$ by an angle of $\pi$. 

We write the tangle obtained by removing two arcs from the diagram $D$ of $\hat{K}$ as the same symbol $D$.
Let a tangle $T_0$ be the tangle sum of $D$ and $D^\ast$.
Then the knot $K$ is regarded as the numerator $N(\tilde{T}+T_0)$ of the tangle sum $\tilde{T}+T_0$ (Figure \ref{diagramofpfofmaintheorem01}).

\begin{figure}[h]
	\centerline{\includegraphics[keepaspectratio]{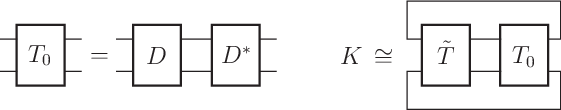}}
\caption{$K=N(\tilde{T}+T_0)$}
\label{diagramofpfofmaintheorem01}
\end{figure}

Applying Conway's fraction formula (see, \cite[Theorem 7.9.1]{Cr2004}, \cite[Theorem 2.3]{Ka1981})
to $\nabla_{N(\tilde{T}+T_0)}(z)$,
we obtain the following formula,
and then $\nabla_K(z)$ is equivalent to the right-hand side of the formula.
\begin{align*}
\nabla_{N(\tilde{T}+T_0)}(z)
=\nabla_{N(\tilde{T})}(z)\nabla_{D(T_0)}(z)+\nabla_{D(\tilde{T})}(z)\nabla_{N(T_0)}(z)
\end{align*}

Where the denominator $D(T_0)$ of $T_0$ is constructed by connecting the left-side endpoints of $T_0$ to each other
and by also connecting the right-side endpoints.

In the above formula, $\nabla_{N(T_0)}(z)=0$ by Lemma \ref{numeratorofT0} below.
Since the knot $D(T_0)$ is equivalent to the connected sum
$\hat{K}\sharp\hat{K}^\ast$ of the partial knot $\hat{K}$
and its mirror image $\hat{K}^\ast$,
it turns out $\nabla_{D(T_0)}(z)=\left(\nabla_{\hat{K}}(z)\right)^2$.
Then we obtain
\begin{align*}
\nabla_K(z)
=\nabla_{N(\tilde{T})}(z)\nabla_{D(T_0)}(z)
=\nabla_{N(\tilde{T})}(z)\left(\nabla_{\hat{K}}(z)\right)^2.
\end{align*}
The above formula means $\Delta_K(t)=\Delta_{K'}(t)\left(\Delta_{\hat{K}}(t)\right)^2$,
so we can conclude Theorem \ref{maintheorem} (1).

The following lemma is a special case of \cite[Theorem 2.3]{L2000}.
\begin{lemma}\label{numeratorofT0}
$\nabla_{N(T_0)}(z)=0$
\end{lemma}
\begin{proof}
Let $L_0=N(T_0)$ be the numerator of $T_0$.
By the symmetry of $T_0$, $L_0$ is a two-component link.
In particular, if the left upper endpoint of the tangle $D$ is connected to the right upper endpoint in $D$,
then the left upper endpoint of $T_0$ is connected to the right upper endpoint in $T_0$.
Similarly, if the left upper endpoint is connected to the right lower endpoint in $D$,
the left upper endpoint of $T_0$ is also connected to the right upper endpoint in $T_0$. 

We fix an orientation of the diagram $D$ and put the opposite orientation on $D^\ast$
in order to provide a natural orientation for the tangle $T_0$.
Let $L_1$ be a link whose orientations of all components are opposite to $L_0$, and $L_2$ the mirror image of $L_1$ (Figure \ref{diagramofnumeratorofT0}).
In Figure \ref{diagramofnumeratorofT0}, $c_i^\ast$ is the crossing point in $D^\ast$ corresponding to a crossing point $c_i$ in $D$. 

\begin{figure}[h]
	\centerline{\includegraphics[keepaspectratio]{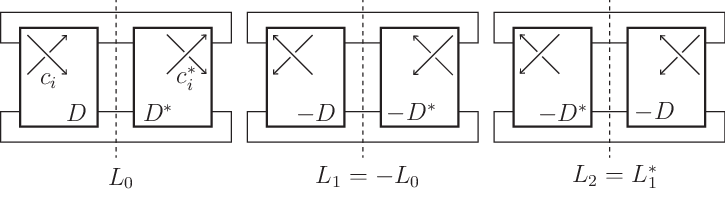}}
\caption{$L_0$, $L_1$, $L_2$}
\label{diagramofnumeratorofT0}
\end{figure}

By properties of Conway polynomial for orientation reversing and mirror images, $\nabla_{L_0}(z)=\nabla_{L_1}(z)=(-1)\nabla_{L_2}(z)$ since the number of components of $L_1$ is two.
Nevertheless, the diagram of $L_2$ is obtained by reflecting the diagram of $L_0$ at the axis indicated by the dotted line in Figure \ref{diagramofnumeratorofT0}.
Then $L_0$ is equivalent to $L_2$ and it follows $\nabla_{L_0}(z)=-\nabla_{L_2}(z)=-\nabla_{L_0}(z)$.
Therefore $\nabla_{L_0}(z)=0$.

\end{proof}

\subsection{Proof of Theorem \ref{maintheorem} (2)}
The proof of Theorem \ref{maintheorem} (2) is almost analogous to the proof of Theorem \ref{symunionlongitude}.
However, the tangle $T$ in Theorem \ref{maintheorem} is not just a twisted tangle.
Therefore, we formulate the proof with more details.

We fix an orientation of $\hat{K}$.
The arcs in diagram $D$ of $\hat{K}$ are labeled $x_1, x_2, \dots, x_m$,
and the corresponding arcs in $D^\ast$ are labeled $x_1^\ast, x_2^\ast, \dots, x_m^\ast$.
The arcs which will be changed to the arcs with endpoints of the tangle $D$ and $D^\ast$ are labeled $x_b$, $x_a$, $x_a^\ast$, $x_b^\ast$ as in the left side of Figure \ref{pfmainthm2_01},
where the each of the arcs $x_b$ and $x_a$ is one of the arcs $x_1, x_2, \dots, x_m$
and also $x_b^\ast$ and $x_a^\ast$ are the corresponding ones.
As in the proof of Theorem \ref{maintheorem}(1), we put the opposite orientation of $D$ on $D^\ast$. 

\begin{figure}[h]
	\centerline{\includegraphics[keepaspectratio]{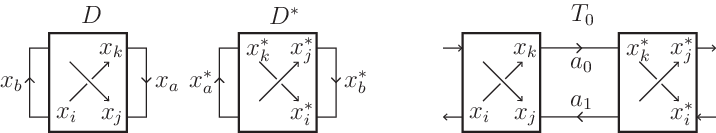}}
\caption{arcs of $D$, $D^\ast$, $T_0$}
\label{pfmainthm2_01}
\end{figure}

In the tangle sum $T_0$, the arcs connecting $D$ and $D^\ast$ are labeled $a_0$ and $a_1$ (the right side of Figure \ref{pfmainthm2_01}).
The arcs in the tangle $T$ are labeled $s_1, s_2, \dots, s_n$, especially the four arcs $s_1, \dots, s_4$ are labeled as in Figure \ref{pfmainthm2_02} which contain the endpoints of $T$.
These arcs are equivalent to some arcs in $D$ and $D^\ast$ as $s_i=x_{\sigma_i}$ for $i=1,2,3,4$, $\sigma_i\in\{1,2,\dots, m\}$.
Then, all arcs in the diagram of $K$ are labeled,
and the orientations of these arcs are induced from a fixed orientation of $\hat{K}$.
We regard these labels as the Wirtinger generators of the knot group $G(K)$.

\begin{figure}[h]
	\centerline{\includegraphics[keepaspectratio]{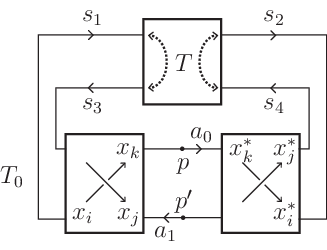}}
\caption{arcs of $D$, $D^\ast$, and $T$, the oriented diagram of $K$}
\label{pfmainthm2_02}
\end{figure}

Let $p$ and $p'$ be points on the arcs $a_0$ and $a_1$ as in Figure \ref{pfmainthm2_02}.
The word $\ell$ of longitude $\lambda_K$ in $G(K)$ is obtained by the following procedure.

Starting from the point $p$ and moving forward with the orientation of $K$,
we add a word $w^\ast=y^\ast {x^\ast}^{-1}$ or $w^\ast={y^\ast}^{-1}x^\ast$ from right side
when we go through each under crossing of the diagram $D^\ast$
according to that crossing is positive or negative,
and also add a word $s_j s_k^{-1}$ or $s_j^{-1} s_k$ of the diagram $T$ as in Figure \ref{pfmainthm2_03}.
The word obtained by repeating this procedure until reaching the point $p'$ is denoted by $\ell_+$.
Similarly, we obtain the word $\ell_-$ by reading words at each under-crossing point, as shown on the left of Figure \ref{pfmainthm2_03}, 
starting at $p$ and moving backward until we reach $p'$.

\begin{figure}[h]
	\centerline{\includegraphics[keepaspectratio]{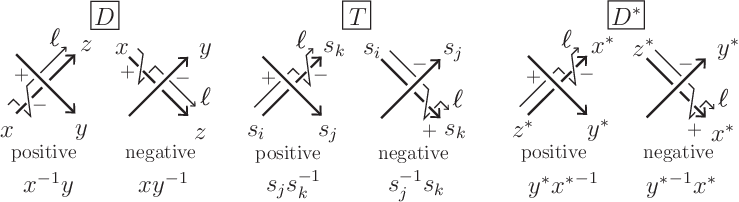}}
\caption{words on each crossing}
\label{pfmainthm2_03}
\end{figure}

Let a word $\ell$ be the concatenation $\ell=\ell_-\cdot\ell_+$.
By the above construction, the loop corresponding to $\ell$ is parallel to $K$, and the linking number between the loop and the knot $K$ is zero,
so $\ell$ is equivalent to the element $\lambda_K$ of the longitude of $K$.

We shall see a precise structure of words in $\ell$ as follows.
Let $r_1^\ast$ be the first under-crossing point encountered when starting from $p$ and moving forward,
and then $r_1$ be the first under-crossing point corresponding to $r_1^\ast$ when moving backward from $p$.
Corresponding to whether $r_1^\ast$ is a positive crossing or negative crossing, the words in $\ell$ corresponding to the crossing $r_1$ and $r_1^\ast$ in the diagrams $D$ and $D^\ast$ are shown as in Figure \ref{pfmainthm2_04}.
 
\begin{figure}[h]
	\centerline{\includegraphics[keepaspectratio]{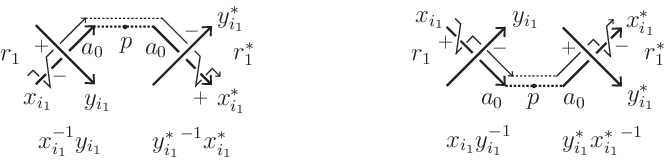}}
\caption{$(r_1, r_1^\ast)=(\text{positive},\text{negative})$ or $(\text{negative},\text{positive})$}
\label{pfmainthm2_04}
\end{figure}

Moving forward and backward, the words $r_k$ and $r_k^\ast$ are similarly obtained as shown in Figure \ref{pfmainthm2_05}.
In the tangle $T$, the words corresponding to the under-crossings
through which we pass are also $s_j s_k^{-1}$ or $s_j^{-1}s_k$
as illustrated in the middle of Figure \ref{pfmainthm2_03}.

\begin{figure}[h]
	\centerline{\includegraphics[keepaspectratio]{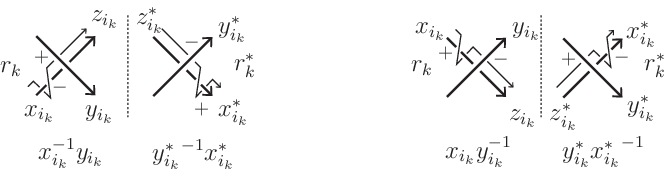}}
\caption{$(r_k, r_k^\ast)=(\text{positive},\text{negative})$ or $(\text{negative},\text{positive})$}
\label{pfmainthm2_05}
\end{figure}

By the above construction, $\ell_+$ and $\ell_-$ have the following form of words.
\begin{align*}
\ell_+=w_1^\ast w_2^\ast \cdots w_i^\ast  S_+ w_{i+1}^\ast \cdots w_m^\ast ,\;\;
\ell_-=w_m\cdots w_{i+1} S_- w_i \cdots w_2w_1
\end{align*}
In the above formula, the pair $(w_k, w_k^\ast)$ is equal to
$(x_{i_k}^{-1}y_{i_k}, {y_{i_k}^\ast}^{-1} x_{i_k}^\ast)$ or
$(x_{i_k} y_{i_k}^{-1}, y_{i_k}^\ast {x_{i_k}^\ast}^{-1})$
for $k=1,2,\cdots m$,
and $S_+$ and $S_-$ are constructed by products of the word $s_j s_k^{-1}$ or $s_j^{-1}s_k$. 

The Wirtinger presentations of the groups $G(\hat{K})$ and $G(K)$ are as follows by using the labels defined before.
\begin{align*}
G(\hat{K})
&=\langle \hat{x}_1, \hat{x}_2, \dots, \hat{x}_m \mid \hat{r}_1, \hat{r}_2, \dots, \hat{r}_m\rangle \\
G(K)
&=\langle x_1, x_2, \dots, x_m, x_1^\ast, x_2^\ast, \dots, x_m^\ast , s_1, s_2, \dots, s_n \mid \\ \relax
& \hspace{72pt} r_1, r_2, \dots, r_m, r_1^\ast, r_2^\ast, \dots, r_m^\ast, r_1^s, r_2^s, \dots, r_n^s \rangle
\end{align*}
The elements $\hat{x}_i$ in $G(\hat{K})$ are the same as the elements $x_i$ in the diagram $D$.
Corresponding to the arcs in Figure \ref{pfmainthm2_02},
some arcs are identified as follows.
\begin{align*}
a_0&=x_{a_0}=x_{a_0}^\ast, a_1=x_{a_1}=x_{a_1}^\ast \\
s_1&=x_{\sigma_1}, s_2=x_{\sigma_2}^\ast, s_3=x_{\sigma_3}, s_4=x_{\sigma_4}^\ast, \;\; \sigma_1,\sigma_2,\sigma_3,\sigma_4 \in \{1,2,\dots m\}
\end{align*} 
The relations $r_i$, $i=1,2,\dots, m$ correspond to the crossing points in the diagram $D$,
and also $r_i^\ast$ in $D^\ast$ corresponds to $r_i$ for $i=1,2,\dots, n$.
The relations $r_i^s$ correspond to the crossing points in $T$.

Next, we define a map $\varphi: G(K)\to G(\hat{K})$ as follows.
\begin{equation*}
\varphi:G(K)\to G(\hat{K}):\;
\left\{\begin{array}{l}
 x_i \mapsto \hat{x}_i,\; x_i^\ast \mapsto \hat{x}_i, \;\; i=1,2,\dots, m \\[5pt]
 s_i \mapsto \hat{x}_b, \;\; i=1,2,\dots, n \\[5pt]
x_{a_0}=x_{a_0}^\ast \mapsto \hat{x}_a,\; x_{a_1}=x_{a_1}^\ast \mapsto \hat{x}_a
\end{array}\right.
\end{equation*}

The relation $\hat{r}_i$ in the Wirtinger presentation of $G(\hat{K})$ is equal to
$\hat{r}_i=\hat{x}_i\hat{x}_j\hat{x}_k^{-1}\hat{x}_j^{-1}$ or $\hat{x}_j\hat{x}_i\hat{x}_j^{-1}\hat{x}_k^{-1}$
for some $i,j,k \in \{1,2,\dots,m\}$
corresponding to the crossing $\hat{r}_i$ is positive or negative.
And also the pair of relations $(r_i, r_i^\ast)$ are equal to 
$(r_i, r_i^\ast)=(x_i x_j x_k^{-1} x_j^{-1}, x_i^\ast x_j^\ast {x_k^\ast}^{-1} {x_j^\ast}^{-1})$ or
$(r_i, r_i^\ast)=(x_j x_i x_j^{-1} x_k^{-1}, x_j^\ast x_i^\ast {x_j^\ast}^{-1} {x_k^\ast}^{-1})$
corresponding to the crossing $r_i$ is positive or negative
as in Figure \ref{pfmainthm2_06}.

\begin{figure}[h]
	\centerline{\includegraphics[keepaspectratio]{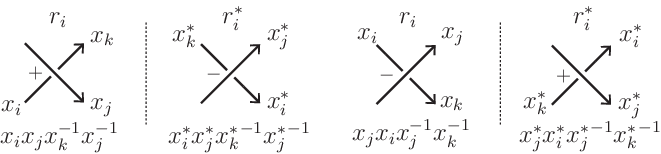}}
\caption{$(r_i, r_i^\ast)=(\text{positive},\text{negative})$ or $(\text{negative},\text{positive})$}
\label{pfmainthm2_06}
\end{figure}

The relation $r_i^s$ is equal to $r_i^s=s_i s_j s_k^{-1} s_j^{-1}$ or $r_i^s=s_j s_i s_j^{-1} s_k^{-1}$
for some $i,j,k\in\{1,2,\dots, n\}$
corresponding to the crossing $r_i^s$ is positive or negative.

The map $\varphi$ maps each of these relations to the corresponding relation $\hat{r}_i$ or the trivial element in $G(\hat{K})$
\begin{align*}
\varphi(r_i)=\varphi(r_i^\ast)
&=\hat{x}_i \hat{x}_j \hat{x}_k^{-1} \hat{x}_j^{-1} \;\;\text{or}\;\; \hat{x}_j \hat{x}_i \hat{x}_j^{-1} \hat{x}_k^{-1} = \hat{r}_i \\
\varphi(r_i^s)
&=\hat{s}\hat{s}\hat{s}^{-1}\hat{s}^{-1}=1.
\end{align*}

Therefore, the map $\varphi$ is a surjective homomorphism.

As we saw before, the word $\ell$ of the longitude has the following form.
\begin{equation*}
\ell = \ell_-\cdot \ell_+
=w_m\cdots w_{i+1}S_- w_i\cdots w_2 w_1 w_1^\ast w_2^\ast \cdots w_i^\ast S_+ w_{i+1}^\ast \cdots w_m^\ast
\end{equation*}
The subword $w_1w_1^\ast$ is equal to
$$w_1w_1^\ast=
x_{i_1}^{-1}y_{i_1}{y_{i_1}^\ast}^{-1}x_{i_1}^\ast \;\;\text{or}\;\; x_{i_1}y_{i_1}^{-1}y_{i_1}^\ast {x_{i_1}^\ast}^{-1}
$$
then $\varphi(w_1w_1^\ast)=1$.
Similarly, we obtain
$\varphi(w_iw_i^\ast)=1$ for $i=2,\dots,m$,
$\varphi(S_-)=\varphi(S_+)=(\text{products of}\;(\hat{s}\hat{s}^{-1})\;\text{or}\;(\hat{s}^{-1}\hat{s}))=1$.

By the arguments above, we can see that $\varphi(\lambda_K)=\varphi(\ell)=1$,
then the proof of Theorem \ref{maintheorem}(2) is completed.

\subsection{Proof of Corollary\ref{coro1.5}}
Let $T$ be a rational tangle represented by a rational number $\beta/\alpha$ which satisfies $\mathrm{gcd}(\alpha, \beta)=1$, $\alpha>0$, and $-\alpha<\beta<\alpha$.
Then, the numerator $N(T)$ is a 2-bridge knot or link $b(\beta/\alpha)$.
If $\alpha$ is even, $T$ satisfies the assumptions of Theorem \ref{maintheorem} that the left-side endpoints are connected in $T$ and the right-side endpoints are also connected.
In this case, there is a continued fraction $[c_1, c_2, \dots, c_r]$ of $\beta/\alpha$
satisfies that all entries $c_i$ are even and the length $r$ is odd  (see \cite[Proposition 12.17]{BZ3rd}),
and $N(T)$ is a 2-bridge knot.
Since a 2-bridge knot $b(\beta/\alpha)$ is fibered if and only if all entries satisfy $|c_i|=2$ \cite[Proposition 12.20]{BZ3rd},
we can construct a non-fibered 2-bridge knot by taking some entries $|c_i|>2$.
Therefore, we can see that there are infinitely many non-fibered 2-bridge knots whose rational tangle $T$ satisfies the assumptions of Theorem \ref{maintheorem}.

Let $T$ be such a rational tangle,
then the numerator $N(T)$ is a non-fibered 2-bridge knot.
Applying the construction of Theorem \ref{maintheorem} to $\hat{K}$ and $T$,
the knot $K$ constructed by way of Theorem \ref{maintheorem} satisfies $K\geq \hat{K}$,
and the epimorphism $\varphi: G(K)\to G(\hat{K})$ maps the longitude of $K$ to the trivial element in $G(\hat{K})$.

Since 2-bridge knots are alternating \cite[Proposition 12.15]{BZ3rd}, a 2-bridge knot is fibered if the Alexander polynomial of the knot is monic \cite{Mu1963}.
Now the 2-bridge knot $N(T)$ is not fibered, then the Alexander polynomial of $N(T)$ is not monic.
Therefore, the Alexander polynomial of $K$ is also not monic, so $K$ is not a fibered knot.
This completes the proof of Corollary \ref{coro1.5}.

\section{Examples}
The complete list of pairs of knots that $K_1 \geq K_2$ up to 10 crossings is presented by Kitano and Suzuki \cite{KS2005}.
The maps that realize these epimorphisms are explicitly presented in \cite{KS2008}. 
In cases when $K_2$ is $3_1$ or $4_1$ in the list,
all of the knots with an epimorphism that maps the longitude of $G(K_1)$ to the trivial element of $G(K_2)$ are listed below:
\begin{align*}
8_{10}, 8_{20}, 9_{24}, 10_{62}, 10_{65}, 10_{77}, 10_{82}, 10_{87}, 10_{99}, 10_{140}, 10_{143} &\geq 3_1 \\
10_{59}, 10_{137} &\geq 4_1
\end{align*}

All knots in the list above except for the knots $10_{82}$ and $10_{87}$
can be constructed by the way suggested in Theorem \ref{maintheorem}.
Furthermore, all knots in the list above except for the knots $10_{82}$, $10_{87}$ and $10_{99}$ are all Montesinos knots \cite{HO1989}, \cite{D2001}.
We take diagrams of these Montesinos knots as shown in Figure \ref{exofmainthm_01}.
The tangles $D$ and $T(\beta/\alpha)$ are rational,
and the denominator $D(D)$ is equivalent to $3_1$ or $4_1$.
A rational number $\beta/\alpha$ presenting a tangle $T(\beta/\alpha)$ is decomposed into an integral part $e$
and a rational part $\beta'/\alpha$, which satisfies $\alpha >0$ and $-\alpha < \beta' < \alpha$.
Following this notation, 
the structures of these Montesinos knots are detailed in Table \ref{tableofexample1},
which can be identified from the diagrams using SnapPy \cite{SnapPy}.

\begin{figure}[h]
	\centerline{\includegraphics[keepaspectratio]{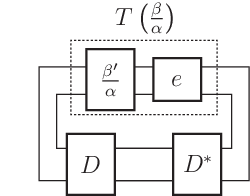}}
\caption{the diagram of a Montesinos knot in the form of Theorem \ref{maintheorem}}
\label{exofmainthm_01}
\end{figure}

\begin{table}[h]
\begin{tabular}{|c||p{48pt}|p{48pt}|p{48pt}|p{48pt}|p{48pt}|}
\hline
$K$ & \hfil $8_{10}$\hfil & \hfil$8_{20}\star$\hfil & \hfil$9_{24}$\hfil & \hfil$10_{62}$\hfil & \hfil$10_{65}$\hfil \\ 
\hline
$(\beta'/\alpha, e)$ & \hfil$(1/2,-2)$\hfil & \hfil$(1/2, 0)$\hfil & \hfil$(-1/2, -2)$\hfil & \hfil$(-1/4, -1)$\hfil & \hfil$(1/4, -2)$\hfil \\
\hline
$N(T(\beta/\alpha))$ & \hfil$3_1$\hfil & \hfil trivial\hfil & \hfil$4_1$\hfil & \hfil$5_1$\hfil & \hfil$5_2$\hfil \\
\hline
\end{tabular}

\vspace{3pt}
\begin{tabular}{|c||p{48pt}|p{48pt}|p{48pt}||p{46pt}|p{48pt}|}
\hline
$K$ & \hfil$10_{77}$\hfil & \hfil$10_{140}\star$\hfil & \hfil$10_{143}$\hfil & \hfil$10_{59}$\hfil & \hfil$10_{137}\star$\hfil \\ 
\hline
$(\beta'/\alpha, e)$ & \hfil$(-1/2, -3)$\hfil & \hfil$(1/4, 0)$\hfil & \hfil$(-3/4, 0)$\hfil & \hfil$(1/2,-2)$\hfil & \hfil$(1/2, 0)$\hfil \\
\hline
$N(T(\beta/\alpha))$ & \hfil$5_2$\hfil & \hfil trivial\hfil & \hfil$3_1$\hfil & \hfil$3_1$\hfil & \hfil trivial\hfil \\
\hline
\end{tabular} \\[3pt]
(knots with the mark $\star$ are a symmetric union of a partial knot) \\[3pt]
\caption{Montesinos knots arising from the construction of Theorem \ref{maintheorem}}
\label{tableofexample1}
\end{table}

The knot $10_{99}$ is constructed as follows.
We take two connected sums $3_1\sharp 3_1^\ast$.
Following the construction in Theorem \ref{maintheorem},
one is placed in the position of the tangle sum $D+D^\ast$,
while the other is positioned in the tangle $T$ as shown on the left of Figure \ref{exofmainthm_02}.
Since the tangle $T$ satisfies the assumption of Theorem \ref{maintheorem}, the knot $10_{99}$ is obtained from the construction in Theorem \ref{maintheorem} and satisfies the properties of its conclusion.
In particular, $\Delta_{10_{99}}(t)=(t^2-t+1)^4=(\Delta_{3_1}(t))^4=(\Delta_{3_1\sharp 3_1^\ast}(t))\cdot(\Delta_{3_1}(t))^2$.

\begin{figure}[h]
	\centerline{\includegraphics[keepaspectratio]{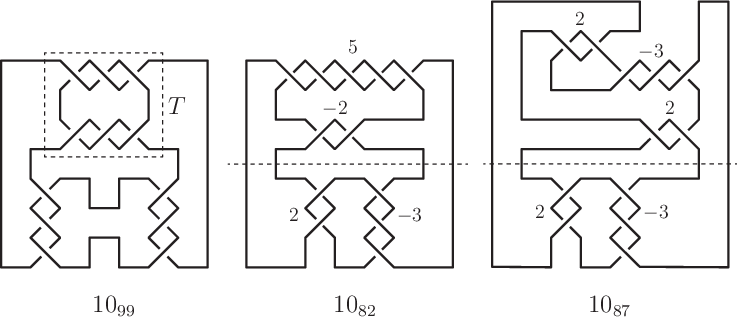}}
\caption{$10_{99}$, $10_{82}$, $10_{87}$}
\label{exofmainthm_02}
\end{figure}

Regarding the remaining knots $10_{82}$ and $10_{87}$,
Masakazu Teragaito has informed us of the following
Proposition \ref{notcomefrommaintheorem} and Lemma \ref{branchedcover}.

\begin{prop}\label{notcomefrommaintheorem}
The knots $10_{82}$ and $10_{87}$ do not have a presentation coming from the construction of Theorem \ref{maintheorem}, for which the partial knot $\hat{K}$ is $3_1$.
\end{prop}

In order to prove Proposition \ref{notcomefrommaintheorem}, 
we first prove the following Lemma.

\begin{lemma}\label{branchedcover}
Either of the double-branched covers along $10_{82}$ and $10_{87}$ is a graph manifold. 
\end{lemma}
\begin{proof}
As in the middle and the right of Figure \ref{exofmainthm_02},
these knots $10_{82}$ and $10_{87}$ are constructed by connecting two Montesinos tangles $M_1$ and $M_2$,
indicated in the form shown on the left of Figure \ref{graphlem_01},
where the Montesinos tangle $M_i$ is a sum of rational tangles, as shown on the right of Figure \ref{graphlem_01}.

\begin{figure}[h]
	\centerline{\includegraphics[keepaspectratio]{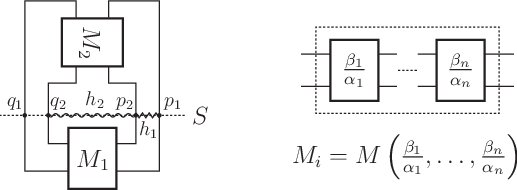}}
\caption{a structure of $10_{82}$ and $10_{87}$, Montesinos tangle $M_i$}
\label{graphlem_01}
\end{figure}

Each Montesinos tangle $M_1$ and $M_2$ is separated by a sphere $S$ in $S^3$.
Let $p_1, p_2, q_1, q_2$ be the intersections of the knot and $S$,
and $h_1$ be an arc connecting $p_1$ and $p_2$, and $h_2$ be an arc connecting $p_2$ and $q_2$ on $S$ as shown in Figure \ref{graphlem_01}.
By the Montesinos trick \cite{Mo1975},
a double-branched cover along a Montesinos tangle
$M\left(\frac{\beta_1}{\alpha_1}, \dots, \frac{\beta_n}{\alpha_n}\right)$
is a Seifert piece $D^2\left(\frac{\beta_1}{\alpha_1}, \dots, \frac{\beta_n}{\alpha_n}\right)$ whose base space is a 2-disk with
$n$ singular fibers.
The lift of $S$ is a torus $\tilde{S}$ with two loops $\tilde{h}_1$ and $\tilde{h}_2$ which are lifts of $h_1$ and $h_2$.
The loops $\tilde{h}_1$ and $\tilde{h}_2$ become regular fibers on each boundary of the Seifert pieces.
We can see that each regular fiber $\tilde{h}_1$ and $\tilde{h}_2$ intersect at one point.
Thus, each regular fiber made from $M_1$ and $M_2$ is not extended into the other Seifert pieces.
It turns out that the Seifert pieces are glued along the essential torus $\tilde{S}$.
Therefore, the double-branched cover along the knot is a graph manifold. 
\end{proof}

\begin{proof}[Proof of Proposition \ref{notcomefrommaintheorem}]
We assume that $10_{82}$ and $10_{87}$ are coming from the construction of Theorem \ref{maintheorem},
where the partial knot $\hat{K}$ is $3_1$.
Then, we can consider the Montesinos tangle $T_0=D+D^\ast=M(1/3,-1/3)$ as shown in Figure \ref{pf_of_prop_01}

\begin{figure}[h]
	\centerline{\includegraphics[keepaspectratio]{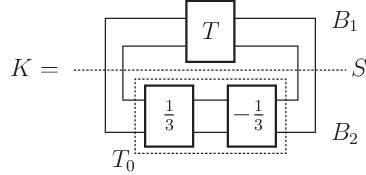}}
\caption{$K=10_{82}$ or $10_{87}$}
\label{pf_of_prop_01}
\end{figure}

Since these knots $10_{82}$ and $10_{87}$ are not Montesinos knots,
these knots are not the numerator of the tangle sum $\tilde{T}+T_0$.
Let $S$ be a separating sphere in $S^3$ as shown in Figure \ref{pf_of_prop_01},
$B_1$ and $B_2$ be 3-balls bounded by $S$,
and $M_K$ be a double branched cover of $S^3$ along $K$.
The lift $\tilde{S}$ of the sphere $S$ is an essential torus in $M_K$,
because the tangles $T$ in $B_1$ and $T_0$ in $B_2$ are not trivial.
Then $M_K$ has a Seifert piece $D^2(1/3, -1/3)$ by a decomposition of the torus $\tilde{S}$.

By Lemma \ref{branchedcover},
$M_K$ does not have a Seifert piece $D^2(1/3, -1/3)$.
It means these knots are not coming from the construction of Theorem \ref{maintheorem}.
\end{proof}

\begin{rem}
It remains to be seen whether
$3_1$ is the only non-trivial knot $\hat{K}$ that satisfies $10_{82}, 10_{87}\geq \hat{K}$. 
\end{rem}

\begin{rem}
A knot formed by summing and gluing some rational tangles is referred to as an arborescent knot \cite{Co1967}.
For the knots constructed in Theorem \ref{maintheorem}, if two tangles $D$ and $T$ are rational,
these knots are considered arborescent knots.
As shown in Figure \ref{exofmainthm_02}, the tangle elements of $10_{82}$ and $10_{87}$ are all rational.
Therefore, these knots are arborescent knots;
and as previously mentioned, they are not Montesinos knots.
The bridge number of these knots is three, according to KnotInfo \cite{KnotInfo}.
Jang classifies all 3-bridge arborescent links \cite{Ja2011},
and these arborescent knots are included in Theorem 1 (1) of her paper. 
\end{rem}

\vspace{8pt}
\noindent\textbf{Acknowledgements:}
The authors would like to express their gratitude to Kouki Taniyama for informing them of Conway's fraction formula, which is used in the proof of Theorem \ref{maintheorem}.
The authors would also like to express their gratitude to Masakazu Teragaito for informing them of Proposition \ref{notcomefrommaintheorem} and Lemma \ref{branchedcover}.
Additionally, they would also like to express their gratitude to Yuta Nozaki for his helpful comments on the draft version of this paper.
The first author is partially supported by JSPS KAKENHI Grant Number 19K03505.
The second author is partially supported by JSPS KAKENHI Grant Number 19K03460.

\end{document}